\documentclass[12pt]{amsart}
\usepackage{latexsym, amsmath,amssymb}
\usepackage[abbrev,backrefs]{amsrefs}
\usepackage{graphicx}
\usepackage{xcolor}

 \numberwithin{equation}{section}

\def\XXint#1#2#3{{\setbox0=\hbox{$#1{#2#3}{%
\int}$ }
\vcenter{\hbox{$#2#3$ }}\kern-.6\wd0}}

\renewcommand{\epsilon}{\varepsilon}
\newtheorem{theorem}{Theorem}

\newtheorem{lemma}[theorem]{Lemma}
\newtheorem{corollary}[theorem]{Corollary}

\newtheorem{proposition}[theorem]{Proposition}
\newtheorem{definition}[theorem]{Definition}
\newtheorem{remark}[theorem]{Remark}
\newcommand{\bth}{\begin{theorem}}
\newcommand{\ble}{\begin{lemma}}
\newcommand{\bcor}{\begin{corr}}

\newcommand{\bdeff}{\begin{deff}}

\newcommand{\bprop}{\begin{proposition}}
\newcommand{\ele}{\end{lemma}}
\newcommand{\ecor}{\end{corr}}
\newcommand{\edeff}{\end{deff}}

\numberwithin{theorem}{section}

\newcommand{\eprop}{\end{proposition}}

\newcommand{\supp}{\text{supp }}
\renewcommand{\Pi}{\varPi}

\renewcommand{\epsilon}{\varepsilon}
\newcommand{\sgn}{{\text {sgn}}}

\usepackage{esint}

\begin{document}

\title[A John-Nirenberg type inequality for critical Sobolev spaces]
{An improvement to the \\ John-Nirenberg inequality for functions\\ in critical Sobolev spaces}

\author[A. D. Mart\'inez]{\'Angel D. Mart\'inez}
\address{Institute for Advanced Study, Fuld Hall 412, 1 Einstein Drive, Princeton, NJ 08540, United States of America} 
\email{amartinez@ias.edu}

\author[D. Spector]{Daniel Spector}
\address{Okinawa Institute of Science and Technology Graduate University,
Nonlinear Analysis Unit, 1919--1 Tancha, Onna-son, Kunigami-gun,
Okinawa, Japan}
\email{daniel.spector@oist.jp}

\begin{abstract}
It is known that functions in a Sobolev space with critical exponent embed into the space of functions of bounded mean oscillation, and therefore satisfy the John-Nirenberg inequality and a corresponding exponential integrability estimate.  While these inequalities are optimal for general functions of bounded mean oscillation, the main result of this paper is an improvement for functions in a class of critical Sobolev spaces.  Precisely, we prove the inequality
\[\mathcal{H}^{\beta}_{\infty}(\{x\in \Omega:|I_\alpha f(x)|>t\})\leq Ce^{-ct^{q'}}\]
for all $\|f\|_{L^{N/\alpha,q}(\Omega)}\leq 1$ and any $\beta \in (0,N]$, where $\Omega \subset \mathbb{R}^N$, $\mathcal{H}^{\beta}_{\infty}$ is the Hausdorff content, $L^{N/\alpha,q}(\Omega)$ is a Lorentz space with $q \in (1,\infty]$, $q'=q/(q-1)$ is the H\"older conjugate to $q$, and $I_\alpha f$ denotes the Riesz potential of $f$ of order $\alpha \in (0,N)$.
\end{abstract}

\maketitle
\section{Introduction}

In 1933, A. Beurling observed that if $f(z)$ is an analytic function for $|z|\leq 1$ and satisfies the integrability condition
\[\int\int_{|z|<1}|f'(z)|^2dxdy\leq\pi,\]
then the set of angles $\theta$ for which $|f(e^{i\theta})|\geq s$ has Lebesgue measure at most $e^{-s^2+1}$ (cf. \cite{Beurling,Moser1971}).  This type of decay implies integrability of $\exp(\alpha|f(z)|^2)$ for $\alpha<1$, while the sharp result, for $\alpha=1$, was not proved until 1985 in the work of Chang and Marshall (cf.~\cite{ChangMarshall}).  This exponential decay estimate of Beurling can be seen as an instance of a Sobolev embedding in the critical exponent with  target a subspace of the space of functions of bounded mean oscillation ($BMO$). Remarkably, this work precedes Sobolev's work on embeddings, the subsequent numerous contributors for results in the critical exponent, and even the introduction of $BMO$.

In this paper we are interested in such improvements to the $BMO$ embedding for functions in a critical Sobolev space.  To this end, let us recall that the space of functions of bounded mean oscillation was introduced by F. John and L. Nirenberg in their seminal paper \cite{JN}. Given $Q_0 \subset \mathbb{R}^N$ a finite cube it can be defined as follows 
\begin{align*}
BMO(Q_0):= \left\{ u \in L^1(Q_0): \|u\|_{BMO(Q_0)}:=\sup_{Q\subset\subset Q_0} \fint_Q |u-u_Q|\;dx<+\infty\right\}.
\end{align*}
Here the supremum is computed with respect to cubes $Q$ with sides parallel to the coordinate axes and $u_Q$ denotes the mean of $u$ in the cube $Q$.  

While the original motivation of such a definition arose from the consideration of problems in elasticity, the influence of this space on harmonic analysis and its applications is far reaching.  A central aspect of this importance is its r\^ole as a replacement for $L^\infty$, e.g. in the theory of concerning the boundedness of translation invariant singular integrals \cite[Theorem 1.1 and Remark 1.3]{Peetre}, as an endpoint in interpolation \cite[Section III]{FeffermanStein}, in the regularity theory of elliptic equations \cite{Moser1960} (which provided a celebrated alternative to the original - and independent - work of De Giorgi and Nash on Hilbert's 19th problem), and, the main point of this paper, as a target for Sobolev embeddings in the critical exponent.

As noted by John and Nirenberg, any bounded function has bounded mean oscillation, though the space $BMO(Q_0)$ is strictly larger than $L^\infty(Q_0)$.  In particular on p.~416 in \cite{JN} they give a class of examples of the form
\begin{align}\label{jnexamples}
u(x):= \int_{Q_0} \log |x-y| f(y)\;dy
\end{align}
for some $f \in L^1(Q_0)$.  The uniting idea for these examples, and a fundamental result proved by them for this space, is that functions in $BMO(Q_0)$ enjoy an exponential decay estimate for their level sets (what is now known as the John-Nirenberg inequality):
\begin{align}\label{jninequality}
|\{x\in Q:|u(x)-u_Q|>t\}|\leq C\exp(-ct/\|u\|_{BMO})|Q|
\end{align}
for certain constants $c,C>0$. From the inequality \eqref{jninequality} one easily deduces an embedding into the Orlicz space of exponentially integrable functions: There exists $c',C'>0$ such that
\begin{align}\label{expint}
\fint_Q\exp\left(c'|u(x)-u_Q|\right)dx\leq C'
\end{align}
for any function $\|u\|_{BMO}\leq 1$. 

The proof of \eqref{jninequality}, and therefore \eqref{expint}, employs a Calder\'on-Zygmund decomposition.  However, as was observed by N. Trudinger \cite[Theorem 1 on p.~476]{Trudinger}, for certain $BMO$ functions one has a simpler proof of the inequality \eqref{expint}.  In particular, for functions which posses a weak derivative in an appropriate space, Trudinger shows how \eqref{expint} follows from representation formulae for these functions in terms of potentials applied to the Euclidean norm of their weak derivatives along with corresponding estimates for these potentials.  This perspective brings into focus two important problems to consider -- that of optimal representations of functions as potentials of their derivatives and that of the mapping properties of these potentials.  The former plays a r\^ole in the determination of sharp constants, e.g.~of the best value of $\beta$ in the inequality
\begin{align*}
\fint_\Omega \exp(\beta |u(x)|^{p'})\;dx \leq C'
\end{align*}
for all $u \in W^{k,p}_0(\Omega)$ such that  $\|\nabla^k u\|_{L^p(\Omega)} \leq 1$, where $p=N/k$ --  in the case $k=1$ the sharp constant is due to J. Moser \cite{Moser1971}, for $k\geq 2$ and $p=2$ it is due to D. Adams \cite{Adams1988}, while for other values of $k,p$ the result is due to I. Shafrir and the second author \cite{ShafrirSpector} (which builds on the foundational work of D. Adams \cite{Adams1988}, see also \cite{GS,GS1} and \cite{Shafrir}). 

\begin{remark}
Although we will not be concerned with this issue in the present work it is worth mentioning that this subtle point has important applications to the Yamabe problem, as Moser did, and other related geometric analysis problems (cf.  \cite{Moser1971, ChangYang} or the survey article \cite{ChangYang2003} and the references therein for more details). Existence of extremizers for the sharp constant in the case $\Omega$ is an $n$-dimensional ball can be found in the work of Carleson and Chang \cite{CarlesonChang}.
\end{remark}
In this paper we focus on the latter question, that of estimates for these potentials in the critical exponent.  In particular, with simple proofs we establish some new exponential decay estimates in the spirit of \eqref{jninequality}.   As we will see, our work extends a result of D. Adams in \cite{Adams1973} and improves upon an estimate of H. Brezis and S. Wainger \cite{BW} (see also \cite{XiaoZhai}).  Here is it useful to change the perspective of the preceding inequality to a corresponding estimate for potentials in the critical exponent:  Let $\Omega \subset \mathbb{R}^N$ be an open and bounded set.  There exists constants $c',C'>0$ such that
\begin{align}\label{expintimproved}
\fint_\Omega \exp\left(c'|I_\alpha f (x)|^{p^\prime}\right)dx\leq C'
\end{align}
for all $f$ with $\operatorname*{supp}f \subset \Omega$ and $\|f\|_{L^p(\Omega)}\leq 1$, where $p=N/\alpha$ and we have used $I_\alpha f$ to denote the Riesz potential of order $\alpha \in (0,N)$ of $f$, defined by
\begin{align*}
I_\alpha f(x) := \frac{1}{\gamma(\alpha)} \int_{\mathbb{R}^N} \frac{f(y)}{|x-y|^{N-\alpha}}\;dy
\end{align*}
for $\gamma(\alpha)=\pi^{N/2}2^{\alpha}\Gamma(\alpha/2)\Gamma((N-\alpha)/2)^{-1}$.  The inequality \eqref{expintimproved} has an extensive history in the literature -- it has been observed by Yudovich in \cite{Yudovich}, is implicit in \cite{P} for $N=2$, the case $\alpha=1$ is proved in Trudinger's paper \cite{Trudinger}, a version for Bessel potentials is due to Strichartz \cite{Strichartz}, the proof of the statement we assert here is in Hedberg's paper \cite{Hedberg}, while the optimal constant was established by Adams in \cite{Adams1988} (see also \cite{XiaoZhai} for the optimal constant on the Lorentz scale).

The consideration of \eqref{jninequality} and \eqref{expint}, along with the comparison of \eqref{expint} and \eqref{expintimproved} prompts one to wonder whether the improved exponential integrability found in \eqref{expintimproved} comes with a corresponding improved exponential decay estimate.  The first result of this paper is the following theorem to this effect.
\begin{theorem}\label{main1}
Let $\Omega \subset \mathbb{R}^N$ be open and $\alpha \in (0,N)$.  There exist constants $c,C>0$ that depend on $\alpha, N,$ and $\Omega$ such that 
\begin{align*}
|\{x\in \Omega :|I_\alpha f(x)|>t\}|\leq Ce^{-ct^{\frac{N}{N-\alpha}}}
\end{align*}
for all $f \in L^{N/\alpha}(\Omega)$ such that $\|f\|_{L^{N/\alpha}(\Omega)}\leq 1$.
\end{theorem}

This is not surprising, since the proof amounts to relaxing the usual arguments in strong-type spaces to a weak-type setting.  However, the technique is interesting as it suggests the possibility of other related inequalities.  Indeed, when one examines the work of Yudovich \cite{Yudovich}, one finds that he asserts the result not only for integrals over a domain, but even for $n$-dimensional hyperplanes intersected with $\Omega$, $n\leq N$ an integer (this bears a resemblance to the special case of Beurling mentioned at the introduction, who obtains the exponential decay on the circle).  This corresponds to a property enjoyed by functions in the critical Sobolev space which is not true for general functions in $BMO$: the trace of such functions are $BMO$ on the restriction.  Our method can be adapted to this setting, and even to the setting of fractal sets.  In order to state our next result let us first introduce the Hausdorff content of a set $E\subset \mathbb{R}^N$ which is defined by
\[\mathcal{H}_{\infty}^{\beta}(E):=\inf\left\{\sum_{i=1}^\infty \omega_{\beta} r_i^{\beta}:E\subset \bigcup_{i=1}^\infty B(x_i,r_i)\right\}.\]
Here the infimum is taken over all possible coverings of arbitrary radii and $\omega_{\beta} := {\pi^{\beta/2}}/{\Gamma\left(\frac{\beta}{2}+1\right)}$ is the volume of a $\beta$-dimensional sphere.

We can now state
\begin{theorem}\label{main}
Let $\Omega \subset \mathbb{R}^N$ be open, $\alpha \in (0,N)$, and $\beta \in (0,N]$.  There exist constants $c,C>0$ that depend on $\alpha, N, \beta,$ and $\Omega$ such that 
\begin{align*}
\mathcal{H}^{\beta}_{\infty}(\{x\in \Omega:|I_{\alpha}f(x)|>t\})\leq Ce^{-ct^{\frac{N}{N-\alpha}}}
\end{align*}
for all $f \in L^{N/\alpha}(\Omega)$ such that $\|f\|_{L^{N/\alpha}(\Omega)}\leq 1$.
\end{theorem}

This result is analogous to an estimate proved by D. Adams in \cite{Adams1973} for the decay of level sets of the convolution of the Bessel kernel and a function in this critical space.  In particular, in the proof of Theorem 3, Item (ii) on p.~913, Adams proves an exponential decay estimate of the $\mu$ measure of the level sets of such a convolution, where $\mu$ is a non-negative Radon measure which satisfies the ball growth condition $\mu(B(x,r)) \leq r^\beta$ for some $\beta>0$.  One finds the analogy in the equivalence of $\mathcal{H}^{\beta}_{\infty}$ and the supremum over all such measures.

From Theorem \ref{main} we deduce the following improvement to the inequality \eqref{expintimproved}, our
\begin{corollary}\label{cor1}
Let $\Omega \subset \mathbb{R}^N$ be open, $\alpha \in (0,N)$, and $\beta \in (0,N]$.  There exist constants $c',C'>0$ that depend on $\alpha, N, \beta,$ and $\Omega$ such that
\begin{align}
\int_\Omega \exp\left(c'|I_\alpha f |^{p^\prime} \right)d\mathcal{H}_\infty^{\beta} \leq C'
\end{align}
for all $f \in L^{N/\alpha}(\Omega)$ such that $\|f\|_{L^{N/\alpha}(\Omega)}\leq 1$.
\end{corollary}

We conclude the introduction with an application of our techniques to improve the dimension on the estimate of H. Brezis and S. Wainger \cite{BW}, which follows from an extension of our Theorem \ref{main} to the Lorentz scale.  To this end, let us recall that Brezis and Wainger \cite[Theorem 3 (ii) on p.~784]{BW} proved a limiting case of a convolution inequality of O'Neil \cite{oneil} which establishes that the second parameter in the Lorentz space $L^{p,q}$ in this critical regime, while microscopic, is magnified in these inequalities\footnote{We here paraphrase the expression used by H. Brezis in \cite{Brezis}, which he conveyed to the second author during a discussion of the r\^ole of the second exponent in the Lorentz spaces.}.  As in this paper we work exclusively with the Riesz kernels, we now state a version of their result in this context\footnote{In \cite{BW}, as before with \cite{Adams1973}, their results are stated for the Bessel kernel.} which has been proved by J. Xiao and Zh. Zhai \cite[Theorem 3.1 (ii) on p.~364]{XiaoZhai}:  There exists constants $c',C'>0$ such that
\begin{align}\label{expintimproved2}
\int_\Omega \exp\left(c'|I_\alpha f (x)|^{q^\prime}\right)dx\leq C'
\end{align}
for all $f$ with $\operatorname*{supp}f \subset \Omega$ and $\|f\|_{L^{N/\alpha,q}(\Omega)}\leq 1$.

In particular, our method firstly enables us to establish an exponential decay of the level sets of such convolutions with respect to the Hausdorff content, which is our
\begin{theorem}\label{main2}
Let $\Omega \subset \mathbb{R}^N$ be open, $\alpha \in (0,N)$, $\beta\in(0,N]$, and $q \in (1,\infty]$.  There exist constants $c,C>0$ that depend on $\alpha, N, \beta,q,$ and $\Omega$ such that 
\[\mathcal{H}^{\beta}_{\infty}(\{x\in \Omega:|I_{\alpha}f(x)|>t\})\leq Ce^{-ct^{q'}}\]
for all $f \in L^{N/\alpha,q}(\Omega)$ such that $\|f\|_{L^{N/\alpha,q}(\Omega)}\leq 1$ and $\operatorname{supp}f\subseteq\Omega$.
\end{theorem}

In the usual way this leads to an improved dimensional version of the inequality of \cite{BW} and \cite{XiaoZhai}, which we state as
\begin{corollary}\label{cor2}
Let $\Omega \subset \mathbb{R}^N$ be open, $\alpha \in (0,N)$, $\beta\in(0,N]$ and $q \in (1,\infty]$.  There exist constants $c',C'>0$ that depend on $\alpha, N, \beta,q,$ and $\Omega$ such that
\begin{align}
\int_\Omega \exp\left(c'|I_\alpha f |^{q^\prime} \right)d\mathcal{H}_\infty^{\beta} \leq C'
\end{align}
for all $f \in L^{N/\alpha,q}(\Omega)$ such that $\|f\|_{L^{N/\alpha,q}(\Omega)}\leq 1$.
\end{corollary}

The paper is divided as follows. In Section \ref{preliminaries}  we provide some preliminaries about the Lorentz spaces we here require.  In Section \ref{lemmas} we prove a variant of a technical result due to Hedberg as well as several other technical lemmata that will be used in the sequel. Finally, in Section \ref{proofs} we prove the main results.  The main point here is to prove Theorem \ref{main2}, as Theorems \ref{main1} and \ref{main} and Corollaries \ref{cor1} and \ref{cor2} will follow as immediate consequences, though we provide proofs for the convenience of the reader.

\section{Preliminaries on Lorentz spaces}\label{preliminaries}

Let us now introduce several equivalent quasi-norms that can be used to define the Lorentz spaces $L^{p,q}(\mathbb{R}^N)$.  We begin with the development of R. O'Neil in \cite{oneil}.  For $f$ a measurable function on $\mathbb{R}^N$, we define
\begin{align*}
m(f,y):= |\{ |f|>y\}|.
\end{align*} 
As this is a non-increasing function of $y$, it admits a left-continuous inverse, called the non-negative rearrangement of $f$, and which we denote $f^*(x)$.  Further, for $x>0$ we define
\begin{align*}
f^{**}(x):= \frac{1}{x}\int_0^x f^*(t)\;dt.
\end{align*}
With these basic results, we can now give a definition of the Lorentz spaces $L^{p,q}(\mathbb{R}^N)$.  
\begin{definition}
Let $1<p<+\infty$ and $1\leq q<+\infty$.  We define
\begin{align*}
\|f\|_{L^{p,q}(\mathbb{R}^N)} := \left( \int_0^\infty \left[t^{1/p} f^{**}(t)\right]^q\frac{dt}{t}\right)^{1/q},
\end{align*}
and for $1\leq p \leq+\infty$ and $q=+\infty$
\begin{align*}
\|f\|_{L^{p,\infty}(\mathbb{R}^N)} := \sup_{t>0} t^{1/p} f^{**}(t).
\end{align*}
\end{definition}
For these Banach spaces, one has a duality between $L^{p,q}(\mathbb{R}^N)$ and $L^{p',q'}(\mathbb{R}^N)$ for $1<p<+\infty$ and $1\leq q < +\infty$ (see, e.g. Theorem 1.4.17 on p.~52 of \cite{grafakos}).  The Hahn-Banach theorem therefore gives
\begin{align*}
\| f\|_{L^{p,q}(\mathbb{R}^N)} = \sup \left\{ \left| \int_{\mathbb{R}^N} fg \;dx \right| : g \in L^{p',q'}(\mathbb{R}^N) \;\; \|g \|_{L^{q',r'}(\mathbb{R}^N)}\leq 1\right\}.
\end{align*}

Let us observe that with this definition
\begin{align*}
\|f\|_{L^{1,\infty}(\mathbb{R}^N)} &= \|f\|_{L^1(\mathbb{R}^N)} \\
\|f\|_{L^{\infty,\infty}(\mathbb{R}^N)} &= \|f\|_{L^\infty(\mathbb{R}^N)},
\end{align*}
where the spaces $L^1(\mathbb{R}^N)$ and $L^\infty(\mathbb{R}^N)$ are intended in the usual sense.  Note that the former equation is not standard, as $L^{1,\infty}(\mathbb{R}^N)$ has another possible definition, which is only possible through the introduction of a different object.  In particular, for $1<p<+\infty$, one has a quasi-norm on the Lorentz spaces $L^{p,q}(\mathbb{R}^N)$ that is equivalent to the norm we have defined.  What is more, this quasi-norm can be used to define the Lorentz spaces without such restrictions on $p$ and $q$.  Therefore let us introduce the following definition.
\begin{definition}
Let $1 \leq p <+\infty$.  If $0<q<+\infty$ we define
\begin{align*}
|||f|||_{\tilde{L}^{p,q}(\mathbb{R}^N)} :=  \left(\int_0^\infty \left(t^{1/p} f^*(t)\right)^{q} \frac{dt}{t}\right)^{1/q},
\end{align*}
while if $q=+\infty$ we define
\begin{align*}
|||f|||_{\tilde{L}^{p,\infty}(\mathbb{R}^N)} :=   \sup_{t>0} t^{1/p} f^*(t).
\end{align*}
\end{definition}
Then one has the following result on the equivalence of the quasi-norm on $\tilde{L}^{p,q}(\mathbb{R}^N)$ and the norm on $L^{p,q}(\mathbb{R}^N)$ (and so in the sequel we drop the tilde):
\begin{proposition}
Let $1<p<+\infty$ and $1\leq q \leq +\infty$.  Then 
\begin{align*}
|||f|||_{\tilde{L}^{p,q}(\mathbb{R}^N)} \leq \|f\|_{L^{p,q}(\mathbb{R}^N)}\leq p' |||f|||_{\tilde{L}^{p,q}(\mathbb{R}^N)}.
\end{align*}
\end{proposition}
The proof for $1 \leq q<+\infty$ can be found as a variation of the one given for Lemma 2.2 in \cite{oneil}, which we record here as our
\begin{lemma}[Hardy's inequality]\label{hardy}
Let $1<p<+\infty$.  Then for any $q \in [1,\infty)$ one has
\begin{align*}
\left(\int_0^\infty \left[x^{1/p} \fint_0^x f(t)\;dt\right]^q \frac{dx}{x}\right)^{1/q} \leq \frac{p}{p-1} \left(\int_0^\infty \left[x^{1/p} f(x)\right]^q \frac{dx}{x}\right)^{1/q}
\end{align*}
\end{lemma}
As the proof cited in \cite{oneil} is a book of Zygmund which does not treat the case $q>p$, we here provide details for the convenience of the reader.
\begin{proof}[Proof of Lemma \ref{hardy}]
By density it suffices to prove the result for functions $f \in C^\infty_c(\mathbb{R}^N)$.  For such functions the fundamental theorem of calculus implies 
\begin{align*}
\int_0^\infty \frac{d}{dx} \left[x^{1/p} \fint_0^x f(t)\;dt\right]^q \;dx=0.
\end{align*}
A computation of the derivative then yields
\begin{align*}
\int_0^\infty q \left[x^{1/p} \fint_0^x f(t)\;dt\right]^{q-1}\left( (1/p-1) x^{1/p-2}\int_0^x f(t)\;dt + x^{1/p-1} f(x)\right)  \;dx = 0,
\end{align*}
or
\begin{align*}
\int_0^\infty \left[x^{1/p} \fint_0^x f(t)\;dt\right]^q \frac{dx}{x} = \frac{p}{p-1} \int_0^\infty \left[x^{1/p} \fint_0^x f(t)\;dt\right]^{q-1}x^{1/p} f(x)  \frac{dx}{x}.
\end{align*}
Letting $I$ to denote the integral on the left-hand-side, Holder's inequality on $(0,\infty)$ equipped with the measure $ \frac{dx}{x}$ with exponents $q,q'$ yields
\begin{align*}
I \leq  \frac{p}{p-1} I^{1-1/q} \left(\int_0^\infty \left[x^{1/p} f(x)\right]^q \frac{dx}{x}\right)^{1/q}
\end{align*}
and the result follows from reabsorbing the term $I^{1-1/q}$.
\end{proof}

It will be useful for our purposes to observe an alternative formulation of this equivalent quasi-norm in terms of the distribution function.  In particular, Proposition 1.4.9 in \cite{grafakos} reads
\begin{proposition}\label{Lorentzequiv}
Let $1\leq p<+\infty$.  If $0<q<+\infty$,  then
\begin{align*}
|||f|||_{L^{p,q}(\mathbb{R}^N)} \equiv p^{1/q} \left(\int_0^\infty \left(t |\{ |f|>t\}|^{1/p}\right)^{q} \frac{dt}{t}\right)^{1/q},
\end{align*}
while if $q=+\infty$
\begin{align*}
|||f|||_{L^{p,\infty}(\mathbb{R}^N)} \equiv  \sup_{t>0} t |\{ |f|>t\}|^{1/p}.
\end{align*}
\end{proposition}

With these definitions, we are now prepared to state a version of H\"older's inequality on the Lorentz scale.  The following theorem is a slight strengthening of the statement in O'Neil's paper \cite[Theorem 3.4]{oneil}, as we observe that one actually can control the norm of the product with the product of the quasi-norms introduced above.
\begin{theorem}\label{holder}
Let $f \in L^{p_1,q_1}(\mathbb{R}^N)$ and $g \in L^{p_2,q_2}(\mathbb{R}^N)$, where
\[\frac{1}{p_1}+\frac{1}{p_2}=\frac{1}{p}\]
and 
\[\frac{1}{q_1}+\frac{1}{q_2}\geq  \frac{1}{q}\]
for some $p>1$ and $q \geq 1$.   Then
\begin{align*}
\|fg\|_{L^{p,q}(\mathbb{R}^N)} \leq e^{1/e} p' |||f |||_{L^{p_1,q_1}(\mathbb{R}^N)}|||g |||_{L^{p_2,q_2}(\mathbb{R}^N)}.
\end{align*}
\end{theorem}

As the paper of O'Neil does not contain a proof and our calculation leads to slightly different quantities and a different constant than the one claimed in his paper, we here provide one for completeness and the convenience of the reader.  To this end, let us recall that O'Neil defines a product operator
\begin{align*}
h=P(f,g)
\end{align*}
as a bilinear operator on two measure spaces with values in a third measure space which additionally satisfies
\begin{align*}
\|h\|_{\infty} &\leq \|f\|_{\infty}\|g\|_{\infty}, \\
\|h\|_{1} &\leq \|f\|_{1}\|g\|_{\infty} 
\end{align*}
and
\[\|h\|_{1} \leq \|f\|_{\infty}\|g\|_{1}.\]
Here $\|\cdot\|_{\infty}$ and $\|\cdot\|_{1}$ denote the essential supremum and the Lebesgue integral on the corresponding measure spaces.  For clarity of exposition we now restrict ourselves to Euclidean space and the notation we have previously introduced.  We note, however, that these results also hold in this more general framework.

For such operators we require the estimate
\begin{lemma}\label{1.5}
\begin{align*}
xh^{**}(x) \leq \int_0^x f^{*}(t)g^{*}(t)\;dt.
\end{align*}
\end{lemma}

Assuming that we have established it, let us deduce Theorem \ref{holder}.
\begin{proof}[Proof of Theorem \ref{holder}]
We have
\begin{align*}
\|h\|_{L^{p,q}(\mathbb{R}^N)}= \left( \int_0^\infty (x^{1/p}h^{**}(x))^q \;\frac{dx}{x}\right)^\frac{1}{q}.
\end{align*}
By Lemma \ref{1.5} one has
\begin{align*}
h^{**}(x) \leq \frac{1}{x} \int_0^x f^{*}(t)g^{*}(t)\;dt.
\end{align*}
which by Hardy's inequality (Lemma \ref{hardy}) implies
\begin{align*}
\|h\|_{L^{p,q}(\mathbb{R}^N)} \leq  p' \left( \int_0^\infty (x^{1/p} f^{*}(x) g^{*}(x))^q \;\frac{dx}{x}\right)^\frac{1}{q}.
\end{align*}
Now if 
\begin{align*}
\frac{1}{p} &= \frac{1}{p_1}+\frac{1}{p_2} \\
\frac{1}{q} &= \frac{1}{q_1}+\frac{1}{q_2} 
\end{align*}
we have
\begin{align*}
x^{1/p} f^{*}(x) g^{*}(x) = x^{1/p_1} f^{*}(x) x^{1/p_2} g^{*}(x)
\end{align*}
and it suffices to apply H\"older's inequality with exponents $q_1,q_2$ to obtain
\begin{align*}
\|h\|_{L^{p,q}(\mathbb{R}^N)} \leq p'|||f|||_{L^{p_1,q_1}(\mathbb{R}^N)} |||g|||_{L^{p_2,q_2}(\mathbb{R}^N)}.
\end{align*}
For any different value of $q$ which is admissible we define
\begin{align*}
\frac{1}{\tilde{q}} &= \frac{1}{q_1}+\frac{1}{q_2}.
\end{align*}
Now Calder\'on's Lemma  implies
\[\|h\|_{L^{p,q}(\mathbb{R}^N)} \leq p' \left( \frac{\tilde{q}}{p}\right)^{1/\tilde{q}-1/q}||h||_{L^{p,\tilde{q}}(\mathbb{R}^N)}\]
for $\tilde{q}\leq q$. This result can be found as Proposition 1.4.10 in \cite{grafakos} or Lemma 2.5 in \cite{oneil} for the alternative norm, but with the same constant. This observation together with the previous case implies
\begin{align*}
\|h\|_{L^{p,q}(\mathbb{R}^N)} &\leq p' \left( \frac{\tilde{q}}{p}\right)^{1/\tilde{q}-1/q}|||f|||_{L^{p_1,q_1}(\mathbb{R}^N)} |||g|||_{L^{p_2,q_2}(\mathbb{R}^N)} \\ 
&\leq e^{1/e}p' |||f|||_{L^{p_1,q_1}(\mathbb{R}^N)} |||g|||_{L^{p_2,q_2}(\mathbb{R}^N)}.
\end{align*}
Notice that the constant can be shown to be $e^{1/e}$, independent of the rest of parameters.
\end{proof}

The proof of Lemma \ref{1.5} will be argued from a variation of O'Neil's Lemma 1.4: 
\begin{lemma}\label{1.4}
If $|f|\leq \alpha$ and the support of $f$ has measure at most $x$ then one has
\begin{align*}
h^{**}(t) &\leq \alpha g^{**}(t) \\
h^{**}(t) &\leq \alpha \frac{x}{t} g^{**}(x).
\end{align*}
\end{lemma}

From this we can prove our Lemma \ref{1.5} as follows.  
\begin{proof}[Proof of Lemma \ref{1.5}]
As in O'Neil's proof of \cite[Lemma 1.5]{oneil} we pick a doubly infinite sequence $\{y_n\}$ such that 
\begin{align*}
y_0&=f^*(t),\\
 y_n&\leq y_{n+1},\\
\lim_{n \to \infty}& y_n = +\infty
 \end{align*}
 and
\[  \lim_{n \to -\infty} y_n = 0.\]
From this we can express
 \begin{align*}
 f(z)= \sum_{n=-\infty}^\infty f_n(z)
 \end{align*}
where
\begin{align*}
f_n(z) :=
  \begin{cases}
                                   0 & \text{if $|f(z)|\leq y_{n-1}$,} \\
                                   f(z)- y_{n-1} \;\sgn \;f(z)  & \text{if $y_{n-1} < |f(z)| \leq y_n$,} \\
  y_{n} \; \sgn \;f(z)  & \text{if $ |f(z)|>y_n$}.
  \end{cases}
\end{align*}
This representation implies
\begin{align*}
h &= P\left(\sum_{n=-\infty}^0 f_n, g\right) + P\left(\sum_{n=1}^\infty f_n, g\right)\\ 
&=: \sum_{n=-\infty}^0 h_n + \sum_{n=1}^\infty h_n 
\end{align*} 
and therefore 
\begin{align*}
h^{**}(t) &\leq \sum_{n=-\infty}^0 h^{**}_n(t) + \sum_{n=1}^\infty h^{**}_n(t).
\end{align*}
For the first we use the top equation in Lemma \ref{1.4} and for the second we use the bottom equation:
\begin{align*}
h^{**}(t) &\leq \sum_{n=-\infty}^0 (y_n-y_{n-1}) g^{**}(t) + \sum_{n=1}^\infty (y_n-y_{n-1}) \frac{m(f,y_{n-1})}{t}g^{**}(m(f,y_{n-1})) \\
&=f^*(t)g^{**}(t) + \frac{1}{t}\int_{f^*(t)}^\infty m(f,y)g^{**}(m(f,y))\;dy \\
&=: I+II
\end{align*}
For the second term we make the change of variables $y=f^*(u)$ to obtain
\begin{align*}
II = -\frac{1}{t}\int_0^t u g^{**}(u)df^*(u).
\end{align*}
The same integration by parts performed in Lemma 1.5 then yields
\begin{align*}
II=-\frac{1}{t} u g^{**}(u)f^*(u) \Bigg|_0^t + \frac{1}{t}\int_0^t g^{*}(u)f^*(u)\;du.
\end{align*}
In particular, for the first term of this second term we find
\begin{align*}
-\frac{1}{t} u g^{**}(u)f^*(u) \Bigg|_0^t = -g^{**}(t)f^*(t),
\end{align*}
which precisely cancels the first term!  Finally the second term is as desired, and thus we obtain the thesis.
\end{proof}

Finally, we complete the proof of our Lemma \ref{1.4}.
\begin{proof}[Proof of Lemma \ref{1.4}]
As in his proof of Lemma 1.4 in \cite{oneil} we define the truncation of the function $g$ at height $u$
\begin{align*}
g_u(z) :=
  \begin{cases}
                                   g(z) & \text{if $|g(z)| \leq u$} \\
                                   u\; \sgn \;g(z) & \text{if $|g(z)|>u$}
  \end{cases}
\end{align*}
and what remains above height $u$, $g^u:=g-g_u$, we find
\begin{align*}
h=P(f,g)&=P(f,g_u)+P(f,g^u) \\
&=h_1+h_2.
\end{align*}
Note that being a product operator implies that if $|f|\leq \alpha$ and has support on a set of measure at most $x$ then
\begin{align*}
\|h_1\|_{L^\infty(\mathbb{R}^N)} &\leq \|f\|_{L^\infty(\mathbb{R}^N)}\|g_u\|_{L^\infty(\mathbb{R}^N)} \leq \alpha u,\\
\|h_1\|_{L^1(\mathbb{R}^N)}  &\leq \|f\|_{L^1(\mathbb{R}^N)}\|g_u\|_{L^\infty(\mathbb{R}^N)} \leq \alpha x u 
\end{align*}
and
\[\|h_2\|_{L^1(\mathbb{R}^N)} \leq  \|f\|_{L^\infty(\mathbb{R}^N)}\|g^u\|_{L^1(\mathbb{R}^N)} \leq \alpha \int_u^\infty m(g,y)\;dy.\]
Thus we estimate 
\begin{align*}
h^{**}(t)=\frac{1}{t}\int_0^t h^*(z)\;dz \leq \frac{1}{t} \int_0^t h^*_1(z)+h^*_2(z)\;dz,
\end{align*}
where we have used
\begin{align*}
h^*(z) \leq h^*_1(z)+h^*_2(z),
\end{align*}
which relies on the fact that $h_1,h_2$ have disjoint support.

Now by the estimates for $h_1,h_2$  (in the first the $L^\infty$ estimate for $h_1$ and in the second the $L^1$ estimate) we find
\begin{align*}
h^{**}(t) \leq  \frac{1}{t} \alpha \left( tu +\int_u^\infty m(g,y)\;dy\right) \\
h^{**}(t) \leq \frac{1}{t}  \alpha\left( xu +\int_u^\infty m(g,y)\;dy\right).
\end{align*}
The choice $u=g^*(t)$ or $g^*(x)$ and the equality \begin{align*}
ag^*(a) +\int_{g^*(a)}^\infty m(g,y)\;dy = a g^{**}(a)
\end{align*}
with $a=t,x$ yield
\begin{align*}
h^{**}(t) \leq  \alpha g^{**}(t) \\
h^{**}(t) \leq \alpha\frac{x}{t} g^{**}(t).
\end{align*}
\end{proof}

Finally we require the $L^1$ endpoint of H\"older's inequality stated in \cite{oneil}.
\begin{theorem}[Theorem 3.5 in \cite{oneil}]\label{holderprime}
Let $f \in L^{p_1,q_1}(\mathbb{R}^N)$ and $g \in L^{p_2,q_2}(\mathbb{R}^N)$, where
\begin{align*}
\frac{1}{p_1}+\frac{1}{p_2}&=1\\
\frac{1}{q_1}+\frac{1}{q_2}&\geq  1
\end{align*}
Then
\begin{align*}
\|fg\|_{L^{1}(\mathbb{R}^N)} \leq e^{1/e} |||f |||_{L^{p_1,q_1}(\mathbb{R}^N)}|||g |||_{L^{p_2,q_2}(\mathbb{R}^N)}.
\end{align*}
\end{theorem}
\begin{proof}
If we again define $h=P(f,g)$ we have
\begin{align*}
\|h\|_{L^1(\mathbb{R}^N)} &= \lim_{x \to \infty} \int_0^x h^*(t)\;dt \\
&= \lim_{x \to \infty} x h^{**}(x) \\
&\leq \lim_{x \to \infty} \int_0^x f^{*}(t)g^{*}(t)\;dt \\
&=\int_0^\infty t^{1/p_1}f^{*}(t) t^{1/p_2}g^{*}(t)\;\frac{dt}{t},
\end{align*}
and H\"older's inequality implies
\begin{align*}
\|h\|_{L^1(\mathbb{R}^N)} \leq |||f |||_{L^{p_1,\tilde{q}_1}(\mathbb{R}^N)}|||g |||_{L^{p_2,q_2}(\mathbb{R}^N)}
\end{align*}
where $\tilde{q}_1$ is chosen such that
\begin{align*}
\frac{1}{\tilde{q}_1}+\frac{1}{q_2}=1.
\end{align*}
The result then follows from Calder\'on's Lemma as in the proof of theorem \ref{holder} above with the same constant $e^{1/e}$.
\end{proof}

\section{Auxiliary results}\label{lemmas}

In this section we expose some technical results that will be used in the sequel. We will need the following estimate for the weak-$L^p$ quasi-norm of truncated potentials:

\begin{lemma}\label{lorentzriesz}
Let $\alpha \in (0,N)$ and $p\in(1,N/\alpha)$.  Then
\begin{align*}
 \left|\left|\left| \frac{\chi_{B(0,r)^c}}{|\cdot|^{N-\alpha}} \right|\right|\right|_{L^{p',\infty}(\mathbb{R}^N)} = |B(0,1)|^{1/p'} r^{-\delta}
 \end{align*}
where $\delta=\frac{N}{p}-\alpha>0$.  
\end{lemma}
\begin{proof}
We begin with the observation that
\begin{align*}
\left|\left\{ x \in B(0,r)^c : \frac{1}{|x|^{N-\alpha}} >t\right\}\right| = 0
\end{align*}
if $t>r^{\alpha-N}$, while in the case $t \leq r^{\alpha-N}$ we have
\begin{align*}
t \left|\left\{ x \in B(0,r)^c : \frac{1}{|x|^{N-\alpha}} >t\right\}\right|^{1/p'} = t^{1-\frac{N}{(N-\alpha)p'}}|B(0,1)|^{1/p'}.
\end{align*}
Therefore we deduce that
\begin{align*}
\sup_{t>0}t \left|\left\{ x \in B(0,r)^c : \frac{1}{|x|^{N-\alpha}} >t\right\}\right|^{1/p'} &= r^{(\alpha-N)\left(1-\frac{N}{(N-\alpha)p'}\right)} |B(0,1)|^{1/p'}\\
&=r^{-\delta}|B(0,1)|^{1/p'},
\end{align*}
which is the desired conclusion.
\end{proof}
\begin{lemma}\label{lorentzholder}
Let $f \in L^{N/\alpha,q}(\Omega)$ and suppose that $\frac{1}{2}\left(N/\alpha+1\right) \leq p<N/\alpha$.  Then $f \in L^{p,1}(\Omega)$ and there exists a constant $C_1=C_1(\alpha,N,q,\Omega)>0$ such that
\begin{align*}
\|f\|_{L^{p,1}(\Omega)} \leq C_1\delta_p^{-1/q'} \|f\|_{L^{N/\alpha,q}(\Omega)} ,
\end{align*}
where $\delta_p:= N/p-\alpha>0$.  Moreover, we can choose $C_1$ such that 
\begin{align*}
C_1\delta_p^{-1/q'} \geq 1.
\end{align*}
\end{lemma}

\begin{proof}
By our slight variation of O'Neil's version of H\"older's inequality in Lorentz spaces, see Theorem \ref{holder} in Section \ref{preliminaries} above, we have
\begin{align*}
\|f\chi_\Omega\|_{L^{p,1}(\mathbb{R}^N)} \leq e^{1/e}p'|||f\chi_\Omega|||_{L^{N/\alpha,q}(\mathbb{R}^N)} |||\chi_\Omega|||_{L^{r,q'}(\mathbb{R}^N)}
\end{align*}
where
\begin{align*}
&\frac{1}{p}=\frac{\alpha}{N}+\frac{1}{r} \\
&1=\frac{1}{q}+\frac{1}{q'}.
\end{align*}
We compute
\begin{align*}
|||\chi_\Omega|||^{q'}_{L^{r,q'}(\mathbb{R}^N)} &= |\Omega|^{q'/r} r \int_0^1t^{q'-1}\;dt \\
&= |\Omega|^{q'/r} \frac{r}{q'},
\end{align*}
which combined with the fact that $r=N/\delta_p$ yields
\begin{align*}
\|f\|_{L^{p,1}(\Omega)} \leq e^{1/e}p'  |\Omega|^{\delta_p/N} \left(\frac{N}{q'\delta_p}\right)^{1/q'} \|f\|_{L^{N/\alpha,q}(\Omega)}.
\end{align*}

Define $p_0:=\frac{1}{2}\left(N/\alpha+1\right)$.  Then the assumption $p_0\leq p$ implies firstly that
\begin{align*}
p' &\leq \frac{\frac{1}{2}\left(N/\alpha+1\right)}{\frac{1}{2}\left(N/\alpha+1\right)-1} \\
&=\frac{N/\alpha+1}{N/\alpha+1-2} \\
&=\frac{N +\alpha}{N-\alpha},
\end{align*}
and so
\begin{align*}
\|f\|_{L^{p,1}(\Omega)} \leq e^{1/e}\frac{N +\alpha}{N-\alpha}  \max\{|\Omega|,1\} \left(\frac{N}{q'}\right)^{1/q'}  \delta_p^{-1/q'} \|f\|_{L^{N/\alpha,q}(\Omega)}.
\end{align*}
Therefore the estimate holds with
\begin{align*}
C_1:=\max\left\{e^{1/e}\frac{N +\alpha}{N-\alpha}  \max\{|\Omega|,1\} \left(\frac{N}{q'}\right)^{1/q'}, \;\;\delta_{p_0}^{1/q'}\right\}.
\end{align*}
\end{proof}

\begin{lemma}\label{lebesgueholder}
Under the hypothesis of Lemma \ref{lorentzholder} let  $f \in L^{s}(\Omega)$ for some $1<s<N/\alpha$. Then
\begin{align*}
\|f\|_{L^{s}(\Omega)} \leq C_1'\|f\|_{L^{N/\alpha,q}(\Omega)},
\end{align*}
where $\eta:= N/s-\alpha>0$ and $C_1'=C_1'(\alpha,N,s,|\Omega|)>0$.
\end{lemma}

\begin{proof}
The proof is analogous to the previous one. Indeed, given $s \in (1,N/\alpha)$ we define $r$ by the relation
\begin{align*}
\frac{1}{s}=\frac{\alpha}{N}+\frac{1}{r}.
\end{align*}
Note that for any choice of $s$ in this range and any $q \in (1,\infty]$ one has
\begin{align*}
\frac{1}{s}\leq\frac{1}{q}+1.
\end{align*}
Therefore we can apply Theorem \ref{holder} to deduce the inequality
\begin{align*}
||| f \chi_\Omega|||_{L^s(\mathbb{R}^N)} \leq \|f\chi_\Omega\|_{L^{s,s}(\mathbb{R}^N)} \leq e^{1/e}s'|||f\chi_\Omega|||_{L^{N/\alpha,q}(\mathbb{R}^N)} |||\chi_\Omega|||_{L^{r,1}(\mathbb{R}^N)}.
\end{align*}
As above
\begin{align*}
|||\chi_\Omega|||_{L^{r,1}(\mathbb{R}^N)} &=r |\Omega|  \int_0^1\;dt \\
&= r |\Omega|,
\end{align*}
and the result follows from the fact that $r=N/\eta$.
\end{proof}

The following estimate is in the spirit of Hedberg's lemma \cite{Hedberg}, while a variant has been argued by Adams in \cite{Adams1975}.

\begin{lemma}[Hedberg]\label{H1}
Under the hypothesis of Lemma \ref{lorentzholder}, for every $\epsilon \in (0,\alpha)$ one has the inequality
\[|I_{\alpha}f(x)|\leq C_4 \mathcal{M}_{\alpha-\epsilon}f(x)^{\frac{\delta}{\delta+\epsilon }}\|f\|_{L^{p,1}(\mathbb{R}^N)}^{\frac{\epsilon }{\epsilon+\delta}}\]
for some $C_4=C_4(N,\alpha,\epsilon)>0$ independent of $\delta$.
\end{lemma}

\begin{proof} 
We begin splitting the Riesz potential in two integrals as follows
\[\begin{aligned}
I_{\alpha}f(x)&=\frac{1}{\gamma(\alpha)}\int_{\mathbb{R}^N}\frac{f(y)}{|x-y|^{N-\alpha}}dy\\
&=\frac{1}{\gamma(\alpha)}\int_{B(x,r)}\frac{f(y)}{|x-y|^{N-\alpha}}dy+\frac{1}{\gamma(\alpha)}\int_{B(x,r)^c}\frac{f(y)}{|x-y|^{N-\alpha}}dy\\
&=J_1(x)+J_2(x).
\end{aligned}\]
We will estimate them separately and will conclude optimizing the choice of the parameter $r$. The first integral can be estimated as follows
\[\begin{aligned}
|J_1(x)|&\leq\frac{1}{\gamma(\alpha)}\sum_{n=0}^{\infty}\int_{B(x,r2^{-n})\setminus B(x,r2^{-n-1})}\frac{|f(y)|}{|x-y|^{N-\alpha}}dy\\
&\leq\frac{1}{\gamma(\alpha)}\sum_{n=0}^{\infty}\frac{(r2^{-n})^N}{(r2^{-n})^{\alpha-\epsilon}}(r2^{-n})^{\alpha-\epsilon}\fint_{B(x,r2^{-n})}|f(y)|\frac{1}{(r2^{-n-1})^{N-\alpha}}dy\\
&\leq r^{\epsilon} \frac{2^{N-\alpha}}{\gamma(\alpha)}\sum_{n=0}^{\infty}(2^{-n})^{\epsilon}\mathcal{M}_{\alpha-\epsilon}(f)(x)\\
&\leq C_2(N,\alpha,\epsilon)r^{\epsilon}\mathcal{M}_{\alpha-\epsilon}f(x),
\end{aligned}\]
where we are using the fractional maximal function, i.e.
\[\mathcal{M}_{\beta}f(x)=\sup_{r>0}r^{\beta}\fint_{B(x,r)}|f(y)|dy.\]
On the other hand, the second integral can be estimated using Theorem \ref{holderprime} (H\"older's inequality in the $L^1$ regime) and Lemma \ref{lorentzriesz}
\[\begin{aligned}
|J_2(x)|&\leq  e^{1/e}\left| \left| \left| \frac{\chi_{B(0,r)^c}}{|\cdot|^{N-\alpha}} \right| \right| \right|_{L^{p',\infty}(\mathbb{R}^N)} |||f|||_{L^{p,1}(\mathbb{R}^N)}\\
&\leq e^{1/e}|B(0,1)|^{1/p'} r^{-\delta}\|f\|_{L^{p,1}(\mathbb{R}^N)} \\
&\leq C_3 r^{-\delta}\|f\|_{L^{p,1}(\mathbb{R}^N)}.
\end{aligned}\]
where 
\begin{align*}
C_3=e^{1/e}\max\{|B(0,1)|,1\}.
\end{align*}
One can then optimize in $r$, however for our purposes simply setting the upper bounds we have proved for $J_1,J_2$ is sufficient.  In particular, from the choice
\[r(x)=\left(\frac{ C_3}{C_2}\frac{\|f\|_{L^{p,1}(\mathbb{R}^N)}}{\mathcal{M}_{\alpha-\epsilon}f(x)}\right)^{\frac{1}{\epsilon +\delta}}\]
one deduces the inequality
\[|I_{\alpha}f(x)|\leq C_4 \mathcal{M}_{\alpha-\epsilon}f(x)^{\frac{\delta}{\delta+\epsilon }}\|f\|_{L^{p,1}(\mathbb{R}^N)}^{\frac{\epsilon }{\epsilon+\delta}}\]
where we have used Young's inequality to estimate \begin{align*}
2C_2^{\frac{\delta}{\delta+\epsilon }}C_3^{\frac{\epsilon}{(\delta+\epsilon)}} \leq 2(C_2+C_3)=:C_4
\end{align*}
which is independent of $\delta$ and \textit{a posteriori} of $p$.
\end{proof}

Let us next recall a weak-type estimate for the fractional maximal function with respect to the Hausdorff content.   
\begin{lemma}\label{FKR}
Let $\gamma\in [0,N)$. There exists a constant $C_5=C_5(N,\gamma)>0$ such that
\[\mathcal{H}_{\infty}^{N-\gamma}(\{x:\mathcal{M}_{\gamma}f(x)>t\})\leq\frac{C_5}{t}\|f\|_{L^1(\mathbb{R}^N)}\]
\end{lemma}
We provide a proof for the convenience of the reader (see also \cite{BagbyZiemer}).

\begin{proof}
Define 
\begin{align*}
E_t:= \{x:\mathcal{M}_{\gamma}f(x)>t\},
\end{align*}
and note that by lower-semicontinuity of the fractional maximal function $E_t$ is an open set.  By the definition of the fractional maximal function, $\mathcal{M}_{\gamma}$, for any $x\in E_t$ there is a radii $r_x$ such that
\begin{align}\label{radiichoice}
r_x^{\gamma}\fint_{B(x,r_x)}|f(y)|dy>t.
\end{align}
Then 
\begin{align*}
E_t \subset \bigcup_{x \in E_t} \overline{B(x,r_x)},
\end{align*}
while the inequality \eqref{radiichoice} implies that
\begin{align*}
\sup_{x \in E_t} r_x<+\infty.
\end{align*}
Therefore, we may apply Vitali's covering theorem (see, e.g. \cite[Theorem 1 on p.~27]{EvansGariepy}) to find a countable subcollection of disjoint balls such that
\begin{align*}
E\subset \bigcup_{i=1}^\infty \overline{B(x_i,5r_i)}
\end{align*}
From this and the definition of the Hausdorff content we find
\begin{align*}
\mathcal{H}_{\infty}^{N-\gamma}(E)&\leq\sum_{i=1}^\infty \omega_{N-\gamma}(5r_i)^{N-\gamma} \\
&\leq \frac{1}{t} \frac{  \omega_{N-\gamma} 5^{N-\beta}}{|B(0,1)|} \sum_{i=1}^\infty \int_{B(x_i,r_i)}|f(y)|dy\\
&\leq\frac{C_5}{t}\|f\|_{L^1(\mathbb{R}^N)},
\end{align*}
the last inequality holds because the selected balls are disjoint.  This completes the proof, with
\begin{align*}
C_5:=\frac{  \omega_{N-\gamma} 5^{N-\beta}}{|B(0,1)|}.
\end{align*}
\end{proof}

\section{Proofs of the Main Results}\label{proofs}

We are now prepared to prove the main result of this paper, Theorem \ref{main2}, from which we will deduce Theorems \ref{main1} and \ref{main}, as well as Corollaries \ref{cor1} and \ref{cor2}.

\begin{proof}[Proof of Theorem \ref{main2}]
We begin with the elementary inequality, for $r>1$, 
\[\mathcal{M}_{\alpha-\epsilon}f(x)\leq (M_{r(\alpha-\epsilon)} |f|^{r}(x))^{1/r}\] 
which together with Lemma \ref{H1} implies that
\begin{align*}
\{x:|I_{\alpha}f(x)|>t\} \subset \{x: t< C_4 \mathcal{M}_{r(\alpha-\epsilon)}|f|^r(x)^{\frac{\delta}{(\delta+\epsilon)r }}\|f\|_{L^{p,1}(\mathbb{R}^N)}^{\frac{\epsilon }{\epsilon+\delta}}\}.
\end{align*}
It is convenient to rewrite this inclusion as
\begin{align*}
\{x:|I_{\alpha}f(x)|>t\} \subset \left\{x: \left(\frac{t}{C_4\|f\|_{L^{p,1}(\mathbb{R}^N)}^{\frac{\epsilon }{\epsilon+\delta}}} \right)^{\frac{(\delta+\epsilon)r}{ \delta}}< \mathcal{M}_{r(\alpha-\epsilon)}|f|^r(x)\right\}
\end{align*}
in order to invoke Lemma \ref{FKR}.  In particular, for any $\beta \in (0,N]$, we may choose $\epsilon \in (0,\alpha]$, $r \in (1,N/\alpha)$ such that $N-\beta=r(\alpha-\epsilon)$ and $N/r-\alpha=\beta/r-\epsilon>0$, from which we deduce
\[\begin{aligned}
\mathcal{H}^\beta_\infty(\{x:|I_{\alpha}f(x)|>t\})&\leq C_5 \left(\frac{C_4\|f\|_{L^{p,1}(\mathbb{R}^N)}^{\frac{\epsilon }{\epsilon+\delta}}}{t} \right)^{\frac{(\delta+\epsilon)r}{ \delta}} \int_{\mathbb{R}^N} |f|^r
\end{aligned}.\]
We recall the fact that $\operatorname*{\supp}f \subset \Omega$ to write $f=f\chi_\Omega$ and utilize Lemma \ref{lorentzholder} to obtain the inequality
\begin{align*}
\|f\|_{L^{p,1}(\Omega)} \leq C_1\delta^{-1/q'} \|f\|_{L^{N/\alpha,q}(\Omega)},
\end{align*}
and Lemma \ref{lebesgueholder} with $s=r>1$ to obtain the inequality
\begin{align*}
\|f\|_{L^{r}(\Omega)} \leq C_1'(\alpha,N,r,|\Omega|) \|f\|_{L^{N/\alpha,q}(\Omega)},
\end{align*}
which combined yield the estimate 
\begin{align*}
\mathcal{H}^\beta_\infty(\{x\in\Omega:|I_{\alpha}f(x)|>t\}) &\leq C_5(C_1')^r\left(\frac{C_4 C_1\delta^{-1/q'}}{t} \right)^{\frac{(\delta+\epsilon)r}{ \delta}}\|f\|_{L^{N/\alpha,q}(\Omega)}^{r+\frac{\epsilon r}{\delta}}\\
&\leq C_6\left(\frac{C_4 C_1\delta^{-1/q'}}{t} \right)^{\frac{(\delta+\epsilon)r}{ \delta}},
\end{align*}
where we have used the fact that $C_1\delta^{-1/q'} \geq 1$, the assumption that $\|f\|_{L^{N/\alpha,q}(\Omega)}\leq 1$, and $C_6:=C_5(C_1')^r$.  

For $t$ sufficiently large, we will choose $\delta=\delta(t)>0$  such that
\begin{align}\label{deltaequation}
\frac{C_4 C_1\delta^{-1/q'} }{t} = \exp(-1).
\end{align}
This is possible whenever $0<\delta(t) \leq \delta_0(\alpha,N,q,|\Omega|,\epsilon)$ with $\delta_0$ chosen sufficiently small, which is to say that $p$ must be chosen sufficiently close to $N/\alpha$.  In particular, recalling $p_0=\frac{1}{2}(N/\alpha+1)$, we can do so for all
\begin{align*}
t \geq t_0:=\frac{C_4 C_1\delta_{p_0}^{-1/q'}}{\exp(-1)}
\end{align*}

For such $t$ this implies
\begin{align*}
\mathcal{H}^\beta_\infty(\{x \in \Omega :|I_{\alpha}f(x)|>t\}) &\leq  C_6\exp \left(-\frac{(\delta+\epsilon)r}{ \delta} \right) \\
&\leq C_6\exp \left(-\frac{\epsilon r}{ \delta} \right).
\end{align*}
The choice of $\delta$ from equation \eqref{deltaequation} thus yields the estimate
\begin{align*}
\mathcal{H}^\beta_\infty(\{x \in \Omega :|I_{\alpha}f(x)|>t\}) 
&\leq C_6\exp(-ct^{q'})
\end{align*}
where
\begin{align*}
c= \epsilon r \left( \frac{\exp(-1)}{C_1C_4} \right)^{q'}.
\end{align*}

This concludes the proof for $t \geq t_0(\alpha, N,q,|\Omega|,\epsilon)$, while in the case $t\in (0,t_0)$ we have
\begin{align*}
\mathcal{H}^\beta_\infty(\{x \in \Omega :|I_{\alpha}f(x)|>t\})  &\leq \mathcal{H}^\beta_\infty(\Omega) \\
&= \mathcal{H}^\beta_\infty(\Omega) \exp(ct_0^{q'})\exp(-ct_0^{q'}) \\
&\leq \mathcal{H}^\beta_\infty(\Omega) \exp(ct_0^{q'})\exp(-ct^{q'}).
\end{align*}
In particular, the theorem holds with $c$ chosen as above and
\[C= \max\{C_6,\mathcal{H}^\beta_\infty(\Omega) \exp(ct_0^{q'})\}.\]

\end{proof}

We next show how one can deduce Theorems \ref{main1} and \ref{main} from Theorem \ref{main2}.
\begin{proof}[Proof of Theorems \ref{main1} and \ref{main}]
First we observe that
\begin{align*}
\|f\|_{L^{N/\alpha,N/\alpha}(\Omega)} \leq \frac{N}{N-\alpha}   ||| f|||_{L^{N/\alpha,N/\alpha}(\Omega)} \equiv  \frac{N}{N-\alpha} \|f\|_{L^{N/\alpha}(\Omega)}
\end{align*}
so that if $\|f\|_{L^{N/\alpha}(\Omega)} \leq 1$, $\|f\|_{L^{N/\alpha,N/\alpha}(\Omega)} \leq \frac{N}{N-\alpha}$.  Therefore by rescaling $f$ by this factor, Theorem \ref{main2} implies
\begin{align*}
\mathcal{H}^\beta_\infty(\{x \in \Omega :|I_{\alpha}f(x)|>t\}) \leq C e^{-\overline{c}t^{\frac{N}{N-\alpha}}}
\end{align*}
with $\overline{c}=\frac{c(N-\alpha)}{N}$.  This completes the proof of Theorem \ref{main}.  Theorem \ref{main1} also follows in the case $\beta=N$, up to a new constant $C$, by the equivalence of $\mathcal{H}^N_\infty$ and the Lebesgue measure $\mathcal{L}^N$.
\end{proof}

We conclude with the proofs of Corollaries \ref{cor1} and \ref{cor2}.

\begin{proof}[Proof of Corollaries \ref{cor1} and \ref{cor2}]
We compute, for $c'>0$ to be determined,
\begin{align*}
\int_\Omega \exp(c'|I_\alpha f|^{q'}) \;d\mathcal{H}^\beta_\infty &= \int_0^\infty \mathcal{H}^\beta_\infty( \{x\in \Omega : \exp(c'|I_\alpha f(x)|^{q'})>t\})\;dt\\
&= \int_0^\infty \mathcal{H}^\beta_\infty\left( \left\{x\in \Omega : |I_\alpha f(x)|>\left(\frac{\ln(t)}{c'}\right)^{1/q'}\right\}\right)\;dt.
\end{align*}

The integral for $t \in (0,1)$ can be estimated above by $\mathcal{H}^\beta_\infty(\Omega)$, while for $t \in (1,\infty)$ we utilize Theorem \ref{main2} to obtain
\begin{align*}
\int_\Omega \exp(c'|I_\alpha f|^{q'}) \;d\mathcal{H}^\beta_\infty &\leq \mathcal{H}^\beta_\infty(\Omega) + \int_1^\infty C\exp\left(-c\frac{\ln(t)}{c'}\right)\;dt\\
&=\mathcal{H}^\beta_\infty(\Omega) +C\int_1^\infty \frac{1}{t^{c/c'}}\;dt<+\infty
\end{align*}
as soon as $c'<c$.  The result follows with
\begin{align*}
C':=\mathcal{H}^\beta_\infty(\Omega) +C\int_1^\infty \frac{1}{t^{c/c'}}\;dt.
\end{align*}
This completes the proof of Corollary \ref{cor2}.  Corollary \ref{cor1} follows with a rescaling of the norm, as computed in the proof of Theorems \ref{main1} and \ref{main}.
\end{proof}

\section{Acknowledgments}
This work was initiated while the first named author was visiting the Nonlinear Analysis Unit in the Okinawa Institute of Science and Technology Graduate University. He warmly thanks OIST for the invitation and hospitality.
The first named author is supported by the National
Science Foundation under Grant No. DMS-1638352.

\end{document}